\documentclass[a4paper, 10pt]{amsart}

\usepackage{amsmath}
\usepackage{amsfonts}
\usepackage{amssymb}
\usepackage[utf8]{inputenc}

\newtheorem{theorem}{Theorem}
\newtheorem{corollary}{Corollary}
\newtheorem{lemma}{Lemma}
\newtheorem{proposition}{Proposition}
\newtheorem{problem}{Problem}
\newtheorem{question}{Question}

\theoremstyle{definition}
\newtheorem{definition}{Definition}

\DeclareMathOperator{\Lin}{Lin}

\newcommand{\NN}{\mathbb{N}}

\newcommand{\EE}{\mathcal{E}}

\title[On the weak and pointwise topologies in function spaces]{On the weak and pointwise topologies in function spaces}
\author{Miko\l aj Krupski}

\address{Institute of Mathematics\\ University of Warsaw\\ \newline Ul. Banacha 2\\02--097 Warszawa\\ Poland }
\email{mkrupski@mimuw.edu.pl}



\date{\today}
\thanks{The author was partially supported by the Polish National Science Center research grant UMO-2012/07/N/ST1/03525}

\begin{document}

\begin{abstract}
For a compact space $K$ we denote by $C_w(K)$ ($C_p(K)$) the space of continuous real-valued functions on $K$ endowed with the weak (pointwise)
topology. In this paper we address the following basic question which seems to be open:
\textit{Suppose that $K$ is an infinite (metrizable) compact space. Is it true that $C_w(K)$ and $C_p(K)$ are homeomorphic?}
We show that the answer is ``no", provided $K$ is an infinite compact metrizable $C$-space. In particular our proof works for any
infinite compact metrizable finite-dimensional space $K$.
\end{abstract}

\maketitle

For a compact space $K$ we can consider three natural topologies on the set $C(K)$ of all continuous real-valued functions on $K$:
the norm topology, the weak topology and the pointwise topology. Let us denote $C(K)$ endowed with the latter two topologies by
$C_w(K)$ and $C_p(K)$ respectively. Suppose that $K$ is an uncountable compact space.
Clearly, the space $C(K)$ equipped with the norm topology is homeomorphic neither to $C_w(K)$ nor to $C_p(K)$: indeed, both $C_w(K)$ and $C_p(K)$
are not metrizable whereas the norm defines a metric on $C(K)$.
For a similar reason, if $K$ is a \textit{countable} compact metrizable space, then $C_w(K)$ is not homeomorphic to $C_p(K)$. In that case $C_p(K)$ is
metrizable and $C_w(K)$ is not.
If we try to compare topologically $C_w(K)$ and $C_p(K)$, for an uncountable compact space $K$, the answer is not obvious at all.
There is a vast literature studying the weak and the pointwise topology in function spaces, but surprisingly it seems to be unknown
whether these two topologies are homeomorphic. More precisely, we can address the following question:
\textit{Let $K$ be an uncountable compact (metrizable) space. Is it true that $C_w(K)$ and $C_p(K)$ are homeomorphic?}
This question seems to be open even for standard uncountable metrizable compacta such as the Cantor space $2^\omega$ or the unit interval $[0,1]$.

It was proved in \cite{C} (cf. \cite{K}) that if $K$ is a finite-dimensional compact metrizable space
then $C_p(K)$ and $C_p([0,1]^\omega)$ are not homeomorphic. On the other hand the
celebrated Miljutin's theorem \cite{Mi} 
asserts that for any two uncountable compact metrizable spaces $K$ and $L$ the spaces $C_w(K)$ and $C_w(L)$
are linearly homeomorphic. The combination of these two results implies immediately that either $C_p(2^\omega)$ is not homeomorphic
to $C_w(2^\omega)$ or $C_p([0,1]^\omega)$ is not homeomorphic to $C_w([0,1]^\omega)$. Similarly,
either $C_p([0,1])$ is not homeomorphic to $C_w([0,1])$ or $C_p([0,1]^\omega)$ is not homeomorphic to $C_w([0,1]^\omega)$, and so on. It is however unclear how
to determine precisely which pairs of spaces are indeed not homeomorphic.

In this short note we show that $C_w(K)$ and $C_p(K)$ are not homeomorphic for any infinite compact metrizable $C$-space $K$ (see Definition \ref{def} below),
in particular, for any infinite
finite-dimensional compact metrizable space $K$. 

Our approach is based on some ideas from \cite{M} (cf. \cite{O} and \cite{K}); however,
to deal with the weak topology on $C(K)$ we consider measures on the compact space $K$ rather than points of that space, as it was done in \cite{M}, \cite{O}, \cite{K}.
One of the key ingredients of the proof is Lemma \ref{lemma} below.

\bigskip

Let $K$ be a compact space.
As usual, we identify the set $C(K)^*$, of all linear functionals on $C(K)$, with
$M(K)$ -- the set of all signed Radon measures on $K$ of finite variation. Using this identification we can equip $M(K)$ with the weak* topology.
For $y\in K$ we denote by $\delta_y\in M(K)$ the corresponding Dirac measure. It is well-known that $K$ can be identified as the subspace
$\{\delta_y:y\in K\}\subseteq M(K)$. If $A\subseteq M(K)$ then $\Lin(A)$ is the linear space spanned by $A$, i.e. the minimal linear subspace of $M(K)$ containing
$A$.

We denote by $\omega$ the set of all non-negative integers, and $\NN=\omega\setminus\{0\}$.
For a natural number $k$ we denote by $[K]^{\leq k}$ ($[K]^{<\omega}$) the hyperspace of all at most $k$-element subsets of $K$ (all finite subsets of $K$)
equipped with the Vietoris topology.

Recall that sets of the form
$$O_K(F;\tfrac{1}{m})=\{f\in C_p(K):\forall x\in F\;\;\lvert f(x) \rvert<\tfrac{1}{m}\},$$
where $F\in[K]^{<\omega}$ and $m\in \NN$, are basic open neighborhoods of the function equal to zero on $K$ in $C_p(K)$.

Similarly, if $F$ is a finite subset of $M(L)$ and $n\in \NN$, then 
$$W_L(F;\tfrac{1}{n})=\{f\in C_w(L):\forall \mu\in F\;\;\lvert \mu(f) \rvert<\tfrac{1}{n}\}$$
is a basic open neighborhood of the function equal to zero on $L$ in $C_w(L)$.
If $F=\{x\}$ or $F=\{\mu\}$ we will write $O_K(x;\tfrac{1}{m})$, $W_L(\mu;\tfrac{1}{m})$ rather than $O_K(\{x\};\tfrac{1}{m})$, $W_L(\{\mu\};\tfrac{1}{m})$.

For $\mu\in M(L)$ and $n\in \NN$ we put
$$\overline{W}_L(\mu;\tfrac{1}{n})=\{f\in C_w(L): \lvert \mu(f) \rvert\leq \tfrac{1}{n}\}.$$
The mappings
\begin{align*}
 &\pi_L:[K]^{\leq k}\times L\to L\\
 &\pi_K:[K]^{\leq k}\times L\to [K]^{\leq k}
\end{align*}
are projections on $L$ and $[K]^{\leq k}$, respectively.

Similarly as in \cite{K}, for a fixed homeomorphism $\Phi:C_p(K)\to C_w(L)$ taking the zero function on $K$ to the zero function on $L$ and for 
$k,m,n\in \NN$, we define the following sets:
\begin{align*}
Z_{k,m,n}&=\{(E,y)\in [K]^{\leq k}\times L\;:\;\Phi(O_K(E;\tfrac{1}{m}))\subseteq \overline{W}_L(\delta_y;\tfrac{1}{n})\},\\
C(k,m,n)&=\pi_{L}(Z_{k,m,n}).
\end{align*}
The following proposition is easy to verify.
\begin{proposition}
The set $Z_{k,m,n}$ is a closed in $[K]^{\leq k}\times L$, for any $k,m,n\in\NN$.
\end{proposition}
\begin{proof}
If $(E,y)\in ([K]^{\leq k}\times L) \setminus Z_{k,m,n}$, then there is $f\in C_p(K)$ such that
$\{f(x):x\in E\}\subseteq (-\tfrac{1}{m}, \tfrac{1}{m})$ and $|\delta_{y}(\Phi(f))|=|\Phi(f)(y)|>\tfrac{1}{n}$.
Obviously, the set
$$\{F\in [K]^{\leq k}:F\subseteq f^{-1}(-\tfrac{1}{m}, \tfrac{1}{m})\} \times \{z\in Y: |\Phi(f)(z)|>\tfrac{1}{n}\}$$
is an open neighborhood of $(E,y)$ in $[K]^{\leq k}\times L$, disjoint from $Z_{k,m,n}$.
\end{proof}
It follows that
$C(k,m,n)$ is closed in $L$ (being a continuous image of a compact set).   
Note that by the continuity of $\Phi$ we have $L=\bigcup_{k,m}C(k,m,n)$. Now, for $m,n,k\in \NN$, we put
$$E(1,m,n)=C(1,m,n)\text { and } E(k,m,n)=C(k,m,n)\setminus C(k-1,m,n), \text { for } k>1.$$
Clearly,
\begin{align}\label{e1}
L=\bigcup_{k,m}E(k,m,n).
\end{align}
For $y\in E(k,m,n)$, let us put
$$\EE(y,m,n)=\pi_K(\pi_L^{-1}(y)\cap Z_{k,m,n}),$$
i.e. $\EE(y,m,n)$ is the family of all exactly $k$-element subsets $E\subseteq K$ satisfying
$\Phi(O_K(E;\tfrac{1}{m}))\subseteq \overline{W}_L(\delta_y;\tfrac{1}{n})$ (this follows from the assumption $y\in E(k,m,n)$).
It is known that for any $y\in E(k,m,n)$ the family $\EE(y,m,n)$ is finite, cf. \cite[Lemma 6.11.9]{vM}. 
Finally, let $\alpha_{m,n}(y)=\bigcup \EE(y,m,n)$, for $y\in E(k,m,n)$.

The following theorem was proved in \cite{M} (cf. \cite[Lemmas 6.11.10, 6.11.1]{vM})
\begin{theorem}\label{partition}
For any $k,m,n\in\NN$ the set $E(k,m,n)$ can be covered by countably many $G_\delta$ (in $L$)
sets $G_r$ such that for each $r\in\mathbb{N}$, there are continuous mappings $f^r_i:G_r\rightarrow K$, $i=1,\ldots,p_r$, such that
$\alpha_{m,1}(y)=\{f^r_1(y),\ldots,f^r_{p_r}(y)\}$
for $y\in G_r$.
\end{theorem}
By \eqref{e1} sets $G_r$ cover the whole space $L$.
Since $\Phi^{-1}:C_w(L)\to C_p(K)$ is continuous, for each $x\in K$ and $m\in \NN$, there is $F_x^m\in [M(L)]^{<\omega}$ and $n\in\NN$ such that
\begin{align}\label{e2}
\Phi^{-1}(W_L(F_x^m;\tfrac{1}{n}))\subseteq O_K(x;\tfrac{1}{m}).
\end{align}



We will need the following lemma.

\begin{lemma}\label{lemma}
If $y\in E(k,m,1)$ for some $k,m\in \NN$
then $\delta_y\in \Lin(\bigcup\{F_x^m: x\in \alpha_{m,1}(y)\}).$
\end{lemma}

\begin{proof}
Suppose that $\delta_y\notin N=\Lin(\bigcup\{F_x^m: x\in \alpha_{m,1}(y)\})$.
By the definition of $\alpha_{m,1}(y)$ there is $A\subseteq\alpha_{m,1}$ such that
\begin{align}\label{e3}
\Phi(O_K(A;\tfrac{1}{m}))\subseteq \overline{W}_L(\delta_y;1).
\end{align}
By our assumption $\delta_y\notin N$ and the separation theorem \cite[Ch. II, 9.2]{S},
there is a linear functional $g:M(L)\to\mathbb{R}$ continuous with respect to the weak* topology in $M(K)$ such that
$$\sup\{g(\mu):\mu\in N\}< g(\delta_y).$$
The weak* continuity of $g$ implies that $g\in C(L)$ (cf. \cite[Theorem 3.16]{F}), i.e. $g(\mu)=\mu(g)$.
Since $N$ is a linear space, scaling $g$ if necessary, we have $g(\delta_y)=\delta_y(g)=g(y)=2$ and $g(\mu)=\mu(g)=0$, for any $\mu \in \bigcup\{F_x^m : x\in A\}$.
By \eqref{e2},
for every $x\in A$, we have $|\Phi^{-1}(g)(x)|<\tfrac{1}{m}$, so $\Phi^{-1}(g)\in O_K(A;\tfrac{1}{m})$. Therefore, by \eqref{e3}
$$g=\Phi(\Phi^{-1}(g))\in \overline{W}_L(\delta_y;1).$$
This means that $|g(y)|=|\delta_y(g)|\leq 1$, a contradiction.
\end{proof}




\begin{definition}\label{def}
A normal space $K$ is called
a \textit{$C$-space} if for any sequence of its open covers $(\mathcal{U}_i)_{i\in\omega}$,
there exists a sequence $(\mathcal{V}_i)_{i\in\omega}$ of families of pairwise disjoint open sets such that $\mathcal{V}_i$ is a refinement of $\mathcal{U}_i$
and $\bigcup_{i\in\omega}\mathcal{V}_i$ is a cover of $K$.
\end{definition}

\begin{definition}\label{def-SID}
A family $\{(A_i, B_i): i\in \omega \}$ of pairs of disjoint closed subsets of a topological space $X$ is
called \textit{essential} if for every family $\{L_i: i \in\omega\}$, where $L_i$ is an arbitrary partition between $A_i$ and $B_i$
for every $i$, we have $\bigcap_{i\in\omega}L_i\neq\emptyset$.
A normal space $X$ is \textit{strongly infinite-dimensional} if it has an infinite essential family of pairs of disjoint closed sets.
\end{definition}

It is well known that any finite-dimensional space, and more generally, any countable-dimensional space
(i.e. a space which is a countable union of finite-dimensional subspaces) is a $C$-space. On the other hand,
a strongly infinite-dimensional space is not a $C$-space. One of the most natural examples of a strongly infinite-dimensional space
is the Hilbert cube $[0,1]^\omega$.

Before we will proceed to the main result of this note, we need to make some preparatory work
concerning strongly infinite-dimensional spaces.
Proposition \ref{SID} and Lemma \ref{lemmaSID} given below are perhaps a part of folklore in the theory of infinite-dimension.
Since we could not find a proper reference in the literature, we shall enclose
a proof here. The reasoning presented below was communicated to the author by R.\ Pol.

\begin{lemma}\label{lemmaSID}
Suppose that $Z$ is a strongly infinite-dimensional compact metrizable space and $Y\subseteq Z$ is $G_\delta$ in $Z$.
Then at least one of the following assertions holds true:
\begin{itemize}
 \item[(a)] $Y$ contains a strongly infinite-dimensional compactum or
 \item[(b)] $Z\setminus Y$ contains a strongly infinite-dimensional compactum.
\end{itemize}
\end{lemma}

\begin{proof}
Since $Y\subseteq Z$ is $G_\delta$, we have $Z\setminus Y=\bigcup_{k=1}^{\infty} F_k$ and each $F_k$ is closed in $Z$ (hence compact).
Fix an infinite essential family $\{(A_i,B_i): i\in\omega\}$ of pairs of disjoint closed subsets of $Z$
(witnessing the fact that $Z$ is strongly infinite-dimensional).
Let $\omega=\bigcup_{k=0}^{\infty} N_k$ be a partition of $\omega$ into infinite, pairwise disjoint sets.

Assume that (b) does not hold true. In particular, for each $k\geq 1$ the set $F_k$ is not strongly infinite-dimensional and
hence, by \cite[Corollary 3.1.5]{vM} there is a sequence $(L_i)_{i\in N_k}$ of partitions in $Z$ between $(A_i,B_i)_{i\in N_k}$
with $(\bigcap_{i\in N_k}L_i)\cap F_k=\emptyset$.

We claim that $\bigcap_{k=1}^{\infty}\bigcap_{i\in N_k}L_i\subseteq Y$ is strongly infinite-dimensional (and hence (a) holds).

Indeed, otherwise there is a sequence $(L_i)_{i\in N_0}$ of partitions in $Z$ between $(A_i, B_i)_{i\in N_0}$ with
$$(\bigcap_{i\in N_0}L_i)\cap\bigcap_{k=1}^{\infty}\bigcap_{i\in N_k} L_i=\emptyset,$$
which is a contradiction with our assumption that the family $\{(A_i,B_i): i\in \omega\}$ is essential.
\end{proof}

\begin{proposition}\label{SID}
Suppose that $X$ is a strongly infinite-dimensional compact metrizable space. Let $X=\bigcup_{n\in \omega}X_n$, where each $X_n$ is a $G_\delta$
subset of $X$. Then, there is $n\in \omega$ such that $X_n$ contains a strongly infinite-dimensional compactum.
\end{proposition}

\begin{proof}
Striving for a contradiction assume that none of $X_n$'s contains a strongly infinite-dimensional compactum. By induction, we construct a decreasing sequence
$F_0\supseteq F_1 \supseteq \ldots \supseteq F_n \supseteq \ldots$ of strongly infinite-dimensional compacta such that, for each $i\in \omega$,
$F_i\subseteq X\setminus X_i$.

For $n=0$ we apply Lemma \ref{lemmaSID} with $Z=X$ and $Y=X_0$. By our assumption (a) does not hold and hence there is a strongly infinite
dimensional compactum $F_0\subseteq X\setminus X_0$.

Assume that, for $n\in \omega$, we already constructed a sequence $F_0\supseteq \ldots \supseteq F_n$ of strongly infinite-dimensional compacta such
that $F_i\subseteq X\setminus X_i$. We apply Lemma \ref{lemmaSID} with $Z=F_n$ and $Y=X_{n+1}\cap F_n$. Again, by our assumption (a) does not hold and consequently
there exists a strongly infinite dimensional compact set $F_{n+1}\subseteq F_n\setminus(X_{n+1}\cap F_n)$. This ends the inductive construction.

Since $(F_n)_{n\in \omega}$ is a decreasing sequence of non-empty compact sets, it has a non-empty intersection $\bigcap_{n\in \omega}F_n$. On the other hand
$\bigcap_{n\in \omega}F_n\subseteq X\setminus \bigcup_{n\in\omega} X_n=\emptyset$, a contradiction.
\end{proof}

Finally, we can prove the following.

\begin{theorem}
If $K$ is a compact metrizable $C$-space, then $C_p(K)$ and $C_w([0,1]^\omega)$ are not homeomorphic.
\end{theorem}

\begin{proof}
Otherwise, there is a homeomorphism $\Phi:C_p(K)\to C_w([0,1]^\omega)$. Since function spaces are homogeneous, we can without loss of generality assume that
$\Phi$ takes the zero function on $K$ to the zero function on $[0,1]^\omega$. By Theorem \ref{partition} we have $[0,1]^\omega=\bigcup_{i\in\NN}G_r$,
where each $G_r$ is a $G_\delta$ subset of $[0,1]^\omega$ and
for every $r\in \NN$, there are continuous mappings $f^r_i:G_r\rightarrow K$, $i=1,\ldots,p_r$, such that
$\alpha_{m,1}(y)=\{f^r_1(y),\ldots,f^r_{p_r}(y)\}$
for $y\in G_r$.

By Proposition \ref{SID} there is $r\in \NN$ such that $G_r$ contains a strongly infinite-dimensional compactum $Q\subseteq G_r$.

Let $f=\bigtriangleup_{i\leq p_r} (f^r_i\upharpoonright Q): Q\to K^{p_r}$ be the restriction to $Q$ of the diagonal mapping, i.e.
$f(y)=(f^r_1(y),\ldots , f^r_{p_r}(y)),$ for $y\in Q\subseteq G_r$.

Since $K^{p_r}$ is a $C$-space (cf. \cite{R}) and $Q$ is not, not all fibers of $f$ are zero-dimensional (in fact not all of them are $C$-spaces),
cf. \cite[5.4]{Fe}.
Hence, there is $x=(x_1,\ldots,x_{p_r})\in K^{p_r}$ such that $f^{-1}(x)$ is uncountable. Note that for any $y\in f^{-1}(x)$ we have
$\alpha_{m,1}(y)=\{x_1,\ldots, x_{p_r}\}$.
Consider $$F_x=\bigcup_{i=1}^{p_r} F_{x_i}^m.$$
Obviously this set is finite.
For $\mu\in M(L)$ let us put $A_\mu=\{y\in L: \mu(\{y\})\neq 0\}$.
For each $\mu\in M(L)$ the set $A_\mu$ is countable being the set of atoms of a measure.
From Lemma \ref{lemma} it follows that for each $y\in f^{-1}(x)$ there is $\mu\in F_x$ such that $y\in A_\mu$.
This means that
$$f^{-1}(x)\subseteq \bigcup_{\mu\in F_x} A_\mu.$$
However, the latter set is countable and thus cannot cover the uncountable fiber $f^{-1}(x)$, a contradiction.
\end{proof}

Combining the above theorem with the Miljutin's theorem \cite{Mi} we get the following.

\begin{corollary}\label{col}
If $K$ is an uncountable compact metrizable $C$-space then $C_w(K)$ and $C_p(K)$ are not homeomorphic.
\end{corollary}

In particular, the above corollary covers the important case of all uncountable finite-dimensional compacta. 

\section*{Open questions}

Though Corollary \ref{col} is quite general, our method does not work for all uncountable metrizable compacta.
Thus we do not know the answer to the following basic question mentioned in the Introduction.

\begin{question}
Suppose that $K$ is an uncountable compact metrizable space (which is not a $C$-space). Is it true that $C_p(K)$ and $C_w(K)$ are not homeomorphic?
\end{question}

It seems that the most interesting particular case of the above question is the following:

\begin{question}
Is it true that $C_p([0,1]^\omega)$ and $C_w([0,1]^\omega)$ are not homeomorphic?
\end{question}

Although we have the proof that, for example, $C_w(2^\omega)$ and $C_p(2^\omega)$ are not homeomorphic our method seems to be fairly complicated.
Moreover it does not provide any topological property distinguishing $C_p(2^\omega)$ and $C_w(2^\omega)$. Thus the following problem seems to
be interesting.
\begin{problem}
Find a topological property distinguishing $C_p(2^\omega)$ and $C_w(2^\omega)$. Find a topological property distinguishing $C_p([0,1])$ and $C_w([0,1])$.
\end{problem}

It is reasonable to ask also what happens outside the metrizable case:

\begin{problem}
Is it true that $C_p(K)$ and $C_w(K)$ are not homeomorphic for any infinite compact space $K$? 
\end{problem}

\subsection*{Acknowledgments}

The author is indebted to Witold Marciszewski for valuable comments and remarks.
This research was partially supported by the Polish National Science Center research grant UMO-2012/07/N/ST1/03525

\end{document}